\documentclass[12pt]{amsart}
\usepackage{enumerate}
\usepackage{helvet,courier}
\usepackage{amsmath}
\usepackage{amssymb}
\usepackage{amsxtra}

\theoremstyle{plain}
\newtheorem{thm}{Theorem}[section]
\newtheorem{lm}[thm]{Lemma}
\newtheorem{cor}[thm]{Corollary}
\newtheorem{prop}[thm]{Proposition}

\theoremstyle{definition}
\newtheorem{rmk}[thm]{Remark}

\newtheorem{example}[thm]{Example}

\theoremstyle{definition}

\newcommand{\R}{\mathbb R}
\newcommand{\Z}{\mathbb Z}
\newcommand{\N}{\mathbb N}
\newcommand{\C}{\mathcal C}

\newcommand{\bnu}{\begin{enumerate}}
\newcommand{\enu}{\end{enumerate}}

\newcommand{\wh}{\widehat}

\newcommand{\bpf}{\begin{proof}}

\newcommand{\epf}{\end{proof}}
\newcommand{\card}{\operatorname{card}}

\allowdisplaybreaks

\numberwithin{equation}{section}

\makeindex  

\begin{document}

\author{Lenka Slav\'ikov\'a}

\address{Department of Mathematical Analysis, Faculty of Mathematics and Physics, Charles University,
	Sokolovsk\'a 83, 186 75 Praha 8, Czech Republic}
\address{Mathematical Institute, University of Bonn, Endenicher Allee 60, 53115 Bonn, Germany}
\email{slavikova@karlin.mff.cuni.cz}


\title[Bilinear Fourier multipliers]{Bilinear Fourier multipliers and the rate of decay of their derivatives}

\begin{abstract}
We investigate two types of boundedness criteria for bilinear Fourier multiplier operators with symbols with bounded partial derivatives of all (or sufficiently many) orders. Theorems of the first type explicitly prescribe only a certain rate of decay of the symbol itself while theorems of the second type require, in addition, the same rate of decay of all derivatives of the symbol. We show that even though these two types of bilinear multiplier theorems are closely related, there are some fundamental differences between them which arise in limiting cases. Also, since theorems of the latter type have so far been studied mainly in connection with the more general class of bilinear pseudodifferential operators, we revisit them in the special case of bilinear Fourier multipliers, providing also some improvements of the existing results in this setting.
\end{abstract}

\maketitle

\section{Introduction and overview of the results}\label{S:introduction}

Assume that $\sigma$ is a bounded function on $\R^{n}$. We denote by $S_\sigma$ the linear multiplier operator defined as
$$
S_\sigma f(x)=\int_{\R^n} \sigma(\xi) \wh{f}(\xi) e^{2\pi i x\cdot \xi}\,d\xi, \quad x\in \R^n,
$$
for any Schwartz function $f$ on $\R^n$. Here, $\wh{f}(\xi)=\int_{\R^n} f(x)e^{-2\pi i x\cdot \xi} \,dx$ stands for the Fourier transform of the function $f$. One of the key questions about the operator $S_\sigma$ is how it acts on different function spaces. Related to this, it is a well-known consequence of the Plancherel identity that $S_\sigma$ admits a bounded extension from $L^2(\R^n)$ to $L^2(\R^n)$. Conversely, for a general bounded function $\sigma$, the operator $S_\sigma$ does not need to be bounded on $L^p(\R^n)$ if $p\neq 2$.  

In connection with various problems involving product-type operations, the bilinear variant of the operator $S_\sigma$ also comes into play. For a given bounded function $m$ on $\R^{2n}$, we define the bilinear multiplier operator $T_m$ as
\begin{equation}\label{E:bilinear_multiplier}
T_m(f,g)(x)=\int_{\R^{n}} \int_{\R^n} m(\xi,\eta) \wh{f}(\xi) \wh{g}(\eta) e^{2\pi i x \cdot (\xi+\eta)}\,d\xi d\eta, \quad x\in \R^n,
\end{equation}
where $f$ and $g$ are Schwartz functions on $\R^n$. In this paper we focus on the study of the $L^2(\R^n) \times L^2(\R^n) \to L^1(\R^n)$ boundedness of the operator $T_m$. While this is, in a sense, a bilinear analogue of the $L^2$-boundedness of the linear operator $S_\sigma$, it is not true that $T_m$ is bounded from $L^2(\R^n) \times L^2(\R^n)$ to $L^1(\R^n)$ for every bounded function $m$. In fact, it was shown by B\'enyi and Torres~\cite{BT} that there is a bounded function $m$ with bounded partial derivatives of all orders such that the associated operator $T_m$ is unbounded from $L^2(\R^n) \times L^2(\R^n)$ to $L^1(\R^n)$. Thus, stronger conditions on the decay of the function $m$ and/or its derivatives than merely their membership to the space $L^\infty(\R^{2n})$ need to be required in order to guarantee the above-mentioned boundedness.

A result of Grafakos, He and Honz\'ik~\cite{GHH} asserts that if a bounded function $m$ has bounded partial derivatives of all orders and is, in addition, square integrable, then $T_m$ is bounded from $L^2(\R^n) \times L^2(\R^n)$ to $L^1(\R^n)$. In a subsequent paper by Grafakos, He and the author~\cite{GHS}, the assumption $m\in L^2(\R^{2n})$ was relaxed to assuming that $m\in L^q(\R^{2n})$ for some $q<4$. It was also pointed out that the same theorem fails if $q>4$; the problem of the validity of the result in the limiting case $q=4$ was however left open. One of the goals of the present paper is to answer this question (see Theorem~\ref{T:q=4} below). 

\medskip
Let us now recall the classical Gagliardo-Nirenberg interpolation inequality, obtained independently by Gagliardo~\cite{G} and Nirenberg~\cite{N}. It asserts, in particular, that if $f$ is a function in $L^{p_1}(\R^{2n})$ whose partial derivatives of an integer order $k$ belong to $L^{p_2}(\R^{2n})$ and $l$ is a positive integer less than $k$, then all partial derivatives of $f$ of order $l$ belong to $L^p(\R^{2n})$, where $p$ is given by
$$
\frac{1}{p}=\frac{1-\theta}{p_1} + \frac{\theta}{p_2}, \quad \theta=\frac{l}{k}.
$$
This inequality was subsequently studied and extended to more general contexts by various authors, see, e.g.,~\cite{BM,BM2,C, CDDD,FFRS,FFRS2}. 

The Gagliardo-Nirenberg interpolation inequality implies that if $m$ is a function in $L^q(\R^{2n})$ which has bounded partial derivatives of all orders, then all partial derivatives of $m$ belong to $L^{r}(\R^{2n})$ for any $r>q$. Therefore, whenever we assume that $m$ is a function on $\R^{2n}$ with bounded partial derivatives of all orders such that $m$ belongs to $L^q(\R^{2n})$ for some $q<4$, then we in fact implicitly assume that all partial derivatives of $m$ belong to $L^r(\R^{2n})$ for some $r<4$. This shows a close relationship between the multiplier theorems from~\cite{GHH, GHS} and a different type of multiplier theorems where a certain rate of decay is prescribed not only for the symbol itself, but also for its derivatives. The latter criteria have been studied mainly in connection with the more general class of pseudodifferential operators (that is, operators of the form~\eqref{E:bilinear_multiplier} with $m$ depending on $x$, $\xi$ and $\eta$). Let us now provide a brief overview of some of these results.


A result by Miyachi and Tomita~\cite{MT}, adapted to the particular case of bilinear multiplier operators, asserts that if the condition
\begin{equation}\label{E:mt}
|\partial_{\xi}^\alpha \partial_{\eta}^{\beta} m(\xi,\eta)| \leq C_{\alpha,\beta} (1+|\xi|+|\eta|)^{-\frac{n}{2}}
\end{equation}
holds for all multiindices $\alpha$, $\beta\in \N_0^n$, then the operator $T_m$ is bounded from $L^2(\R^n) \times L^2(\R^n) \to L^1(\R^n)$. In addition, the same conclusion fails to be true if the power $n/2$ on the right-hand side of~\eqref{E:mt} is replaced by any smaller power.

The previous result was further improved in the recent paper by Kato, Miyachi and Tomita~\cite{KMT}, and in the subsequent paper~\cite{KMT2} by the same authors. They obtained a sufficient condition for the $L^2(\R^n) \times L^2(\R^n) \to L^1(\R^n)$ boundedness of $T_m$ of the form
\begin{equation}\label{E:kmt}
|\partial_{\xi}^\alpha \partial_{\eta}^{\beta} m(\xi,\eta)| \leq C_{\alpha,\beta} V(\xi,\eta),
\end{equation}
where $\alpha$, $\beta \in \N^n_0$ are multiindices satisfying $|\alpha| \leq \lfloor n/2 \rfloor +1$ and $|\beta| \leq \lfloor n/2 \rfloor +1$ and $V$ is a non-negative function on $\R^{2n}$ fulfilling
\begin{equation}\label{E:V}
\card\{(k,l) \in \Z^{n} \times \Z^n:~V(k,l)>\lambda\} \lesssim \lambda^{-4}, \quad \lambda>0,
\end{equation}
and having a ``moderate behavior'' (this means, roughly speaking, that $V(\xi_1,\eta_1)$ is comparable to $V(\xi_2,\eta_2)$ if $(\xi_1,\eta_1)$ is close to $(\xi_2,\eta_2)$; see~\cite[Definition 3.7]{KMT} for a precise definition). The authors also proved a version of this result involving a Besov-type space of smoothness $n/2$ in each of the variables $\xi$ and $\eta$. 


\medskip
The goal of this paper is to investigate the two types of multiplier theorems mentioned above and, in particular, to point out the differences between them that arise in limiting cases. 

We start by revisiting the results of~\cite{KMT}.
Unlike in~\cite{KMT}, where the general case of bilinear pseudodifferential operators was studied, we restrict ourselves to the special case of bilinear multiplier operators and we  obtain some improvements of the previous result in this setting. Namely, we lower the total number of derivatives required in condition~\eqref{E:kmt} from $n+\varepsilon$ to $n/2+\varepsilon$. In addition, we show that the assumption that the function $V$ in~\eqref{E:kmt} has moderate behavior can be omitted, providing that condition~\eqref{E:V} is replaced by
$$
|\{(\xi,\eta) \in \R^{n} \times \R^n:~V(\xi,\eta)>\lambda\}| \lesssim \lambda^{-4}, \quad \lambda>0.
$$
The proof of our result is independent of that in~\cite{KMT}, using the decomposition of the multiplier in terms of product-type wavelets (an approach inspired by the papers~\cite{GHH}, \cite{GHH2}, \cite{GHS}) and combinatorial arguments. 

To formulate our results we need the notion of the fractional Laplace operator, which is defined for any $s>0$ as 
$$
[(I-\Delta)^{\frac{s}{2}} f]^{\wh{ ~}} (\xi) = (1+4\pi^2 |\xi|^2)^{\frac{s}{2}} \wh{f}(\xi), \quad \xi \in \R^{2n}.
$$

Our first main result is the following theorem.

\begin{thm}\label{T:sobolev}
Let $m$ be a function on $\R^{2n}$ 
satisfying
\begin{equation}\label{E:C_s(m)}
C_s(m):=\sup_{\lambda>0} \lambda |\{x\in \R^{2n}:~ |(I-\Delta)^{\frac{s}{2}} m(x)| >\lambda\}|^{\frac{1}{4}} <\infty
\end{equation}
for some $s>n/2$.
Then the associated operator $T_m$ admits a bounded extension from $L^2(\R^n) \times L^2(\R^n)$ to $L^1(\R^n)$ and
$$
\|T_m(f,g)\|_{L^1(\R^n)} \leq C C_s(m) \|f\|_{L^2(\R^n)} \|g\|_{L^2(\R^n)}.
$$
\end{thm}

In addition, we obtain a version of this result involving a Besov-type space of smoothness $n/2$; see Section~\ref{S:multiplier_theorem} for more details. We also point out that it is a consequence of Theorem~\ref{T:sobolev} that if a function $m$ has partial derivatives of order up to $\lfloor n/2 \rfloor +1$ in $L^4(\R^{2n})$ then the associated operator $T_m$ is bounded from $L^2(\R^n) \times L^2(\R^n)$ to $L^1(\R^n)$.

A crucial tool for proving Theorem~\ref{T:sobolev} is an elementary lemma about expressing a given (say non-negative) function defined on $\Z^{2n}=\Z^{n} \times \Z^n$ as a sum of two functions with disjoint supports, the first one having uniformly bounded $\ell^2$-norms over all rows and the second one having uniformly bounded $\ell^2$-norms over all columns. We discuss this problem in Section~\ref{S:lemma} where we show that such a decomposition is indeed possible if the function satisfies the estimate~\eqref{E:V}, and that no condition on the cardinality of level sets of that function weaker than~\eqref{E:V} can guarantee the existence of such a decomposition. 

We now recall that it follows from the results of~\cite{MT} described above that the conclusion of Theorem~\ref{T:sobolev} fails if the power $1/4$ in~\eqref{E:C_s(m)} is replaced by any smaller power. In fact, more is true; namely, \eqref{E:C_s(m)} is the weakest condition on the decay of the measure of level sets of partial derivatives of $m$ that guarantees the $L^2(\R^n) \times L^2(\R^n) \to L^1(\R^n)$ boundedness of the operator $T_m$. This is the content of the next theorem. It shows, in particular, an intimate connection of this problem with the above-mentioned combinatorial question. We note that when optimality of condition~\eqref{E:C_s(m)} is discussed, we only consider the behavior of the function 
$$
\lambda \mapsto |\{x\in \R^{2n}:~ |(I-\Delta)^{\frac{s}{2}} m(x)| >\lambda\}|
$$
for $\lambda$ near zero. This is a natural assumption, considering that a typical function $m$ for which Theorem~\ref{T:sobolev} is of interest has bounded partial derivatives of all (or sufficiently many) orders. 

\begin{thm}\label{T:sharpness}
Assume that $\mu: (0,1) \rightarrow [0,\infty)$ is a non-increasing function satisfying
$$
\sup_{\lambda \in (0,1)} \lambda (\mu(\lambda))^{\frac{1}{4}} = \infty.
$$
Given $s\in 2\N$, there is a smooth function $m$ on $\R^{2n}$ with bounded partial derivatives of all orders which fulfills
$$
|\{x\in \R^{2n}:~ |(I-\Delta)^{\frac{s}{2}} m(x)| >\lambda\}| \leq \mu(\lambda), \quad \lambda \in (0,1)
$$
and for which the associated operator $T_m$ is unbounded from $L^2(\R^n) \times L^2(\R^n)$ to $L^1(\R^n)$.
\end{thm}

Next, we focus on the other type of boundedness criteria, where a certain rate of decay is explicitly prescribed only for the symbol itself, but not for its derivatives. We show that, in contrast to Theorem~\ref{T:sobolev}, this criterion fails in the limiting case $q=4$. 

\begin{thm}\label{T:q=4}
There is a smooth bounded function $m$ on $\R^{2n}$ belonging to the space $L^4(\R^{2n})$ which has bounded partial derivatives of all orders and for which the associated operator $T_m$ is unbounded from $L^2(\R^n) \times L^2(\R^n)$ to $L^1(\R^n)$.
\end{thm}

We point out that it is unclear what is the weakest condition on the decay of the measure of level sets of the symbol $m$ in order to guarantee the $L^2(\R^n) \times L^2(\R^n) \to L^1(\R^n)$ boundedness of the associated operator $T_m$ under the assumption that $m$ has bounded partial derivatives of all orders; in fact, it is not even clear whether such a condition exists at all. To illustrate this further, we observe in Section~\ref{S:q=4} that the function $m$ constructed in the proof of Theorem~\ref{T:q=4} satisfies the estimate
\begin{equation}\label{E:log}
|\{(\xi,\eta)\in \R^{n}\times \R^n:~ |m(\xi,\eta)| >\lambda\}|\lesssim \lambda^{-4} \log^{-\alpha}(e/\lambda), \quad \lambda \in (0,1),
\end{equation}
for every $\alpha >0$. Thus, no condition of the form~\eqref{E:log} is sufficient in this setting. On the other hand, we have the sufficiency of the condition
$$
|\{(\xi,\eta)\in \R^{n}\times \R^n:~ |m(\xi,\eta)| >\lambda\}|\lesssim \lambda^{-q}, \quad \lambda >0,
$$
for any $q<4$; see~\cite[Theorem 1.3]{GHS}.

Let us finally mention that the class of multipliers covered by Theorem~\ref{T:sobolev} is wider than the one from~\cite[Theorem 1.3]{GHS}. A similar conclusion was established in~\cite{KMT} for symbols with unlimited smoothness. We extend this comparison result to symbols of limited smoothness by making use of the fractional variant of the Gagliardo-Nirenberg interpolation inequality obtained by Brezis and Mironescu~\cite{BM}. More details are given in Section~\ref{S:comparison}, where we also briefly comment on the relationship of Theorem~\ref{T:sobolev} with bilinear multiplier theorems of H\"ormander type and with the multiplier theorem~\cite[Theorem 1.1]{GHS}, which has applications for proving boundedness of rough bilinear singular integral operators.

\medskip

{\it Notation.} Let us now fix notation that will be used throughout the paper. Given $n\in \N$ and $x\in \R^n$, we denote by $|x|$ and $|x|_\infty$ the Euclidean and maximum norm of $x$, respectively. Also, given a real number $y$, $\lfloor y \rfloor$ stands for the integer part of $y$, that is, the largest integer which does not exceed $y$. 

For $p\in [1,\infty)$, we denote by $L^p(\R^n)$ the space of all Lebesgue measurable functions on $\R^n$ whose absolute value is integrable when raised to the power $p$, while $L^\infty(\R^n)$ is the space of all essentially bounded functions on $\R^n$. Variants of these spaces when the Lebesgue measure on $\R^n$ is replaced by the counting measure on some countable set $\C$ are denoted by $\ell^p(\C)$. The Schwartz space of rapidly decreasing functions on $\R^n$ is denoted by $\mathcal S(\R^n)$, and $\mathcal S'(\R^n)$ stands for its dual, the space of tempered distributions. By $\left<m,f\right >$ we mean the action of a temperate distribution $m$ on a function $f\in \mathcal S(\R^n)$. The Fourier transform of a temperate distribution $m$ is denoted by $\wh{m}$, and the inverse Fourier transform of a function $f\in \mathcal S(\R^n)$ is denoted either by $\check{f}$, or by $\mathcal F^{-1}(f)$. 

By writing $``\lesssim``$ we mean that the expression on the left-hand side of $``\lesssim``$  is dominated by the expression on the right-hand side up to multiplicative constants depending only on unessential quantities. The relation $``\approx``$ between two expressions means that they are bounded by each other up to multiplicative constants depending only on unessential quantities.

\section{An elementary lemma}\label{S:lemma}

In this section we prove an elementary combinatorial lemma which will serve as a crucial tool for proving Theorem~\ref{T:sobolev}. 

Let us start by introducing some necessary terminology. For a totally $\sigma$-finite measure space $(R,\nu)$, we define the Lorentz space $L^{4,\infty}(R,\nu)$ as the collection of all $\nu$-measurable functions $f$ on $R$ satisfying 
$$
\|f\|_{L^{4,\infty}(R,\nu)} = \sup_{\lambda >0} \lambda \nu(\{x\in R: |f(x)|>\lambda\})^{\frac{1}{4}}<\infty.
$$
Alternatively, the quantity $\|f\|_{L^{4,\infty}(R,\nu)}$ can be expressed as
$$
\|f\|_{L^{4,\infty}(R,\nu)} = \sup_{t>0} t^{\frac{1}{4}} f^*_\nu(t),
$$
where 
$$
f^*_\nu(t)=\inf\{\rho \geq 0: ~ \nu(\{x\in R: ~|f(x)|>\rho\}) < t\}, \quad t>0,
$$
stands for the non-increasing rearrangement of $f$ with respect to the measure $\nu$. 

In the special case when $(R,\nu)$ is the Euclidean space $\R^{2n}$ equipped with the $2n$-dimensional Lebesgue measure $\lambda_{2n}$, we write $L^{4,\infty}(\R^{2n})$ instead of $L^{4,\infty}(\R^{2n},\lambda_{2n})$ for simplicity. In addition, if $(R,\nu)$ is a countable set $\C$ equipped with the counting measure then the corresponding Lorentz space is denoted by $\ell^{4,\infty}(\C)$. We shall also skip the subscript $\nu$ in the notation for the non-increasing rearrangement of a function if no confusion can arise about the choice of the underlying measure.

\begin{lm}\label{L:decomposition}
Let $\C$ be a countable set, and let $f$ be a function on $\C \times \C=\C^2$ such that $f\in \ell^{4,\infty}(\C^2)$. Then we can write $\C^2$ as a union of two disjoint sets $S_1$ and $S_2$ satisfying
\begin{equation}\label{E:1}
\left(\sum_{l:~(k,l)\in S_1} |f(k,l)|^2\right)^{\frac{1}{2}} \leq C \|f\|_{\ell^{4,\infty}(\C^2)} \quad \text{for every $k\in \C$}
\end{equation}
and
\begin{equation}\label{E:2}
\left(\sum_{k:~(k,l)\in S_2} |f(k,l)|^2\right)^{\frac{1}{2}} \leq C \|f\|_{\ell^{4,\infty}(\C^2)} \quad \text{for every $l\in \C$},
\end{equation}
where $C>0$ is an absolute constant. 
\end{lm}

\begin{proof}
We can assume, without loss of generality, that $\|f\|_{\ell^{4,\infty}(\C^2)}=1$. Then 
$$
\card\{(k,l) \in \C^2:~|f(k,l)|>\lambda\} \leq \lambda^{-4}, \quad \lambda>0.
$$
In particular, if $\lambda>1$ then $\card\{(k,l) \in \C^2:~|f(k,l)|>\lambda\} <1$, which implies that the corresponding level sets are empty, and thus $\|f\|_{\ell^{\infty}(\C^2)}\leq 1$. 

We first construct two auxiliary subsets $\tilde{S_1}$ and $\tilde{S_2}$ of $\C^2$ (not necessarily disjoint).
Let us fix $k\in \C$. If $\sum_{l\in \C} |f(k,l)|^2 \leq 2$ then we set $(k,l)\in \tilde{S_1}$ for every $l\in \C$. Conversely, if $\sum_{l\in\C} |f(k,l)|^2 > 2$ then we rearrange those numbers $|f(k,l)|$ which are positive in decreasing order: 
$$
|f(k,l_1)|\geq |f(k,l_2)| \geq |f(k,l_3)| \geq \dots,
$$
where $l_1, l_2, \dots$ is a suitable permutation of (a subset of) $\C$.
Note that such a rearrangement can be constructed since all level sets $\{(k,l) \in \C^2: ~|f(k,l)|>\lambda\}$ with $\lambda>0$ are finite. Then we find $N_k\in \N
$, $N_k\geq 2$ satisfying $\sum_{i=1}^{N_k-1} |f(k,l_i)|^2 <2$ and  $\sum_{i=1}^{N_k} |f(k,l_i)|^2 \geq 2$ (notice that $N_k\geq 2$ because $\|f\|_{\ell^\infty(\C^2)}\leq 1$). Then we set $(k,l_i)\in \tilde{S_1}$ for $i=1,\dots, N_k$, and also $(k,l) \in \tilde{S_1}$ if $|f(k,l)|=0$. We observe that
\begin{equation}\label{E:S_1}
\sum_{l:~(k,l)\in \tilde{S_1}} |f(k,l)|^2
=|f(k,l_{N_k})|^2 + \sum_{i=1}^{N_k-1} |f(k,l_i)|^2 
\leq 1+2=3.
\end{equation}

Let us now fix $l\in \C$. If $\sum_{k\in \C} |f(k,l)|^2 \leq 2$ then we set $(k,l)\in \tilde{S_2}$ for every $k\in \C$. Conversely, if $\sum_{k\in\C} |f(k,l)|^2 >2$ then we rearrange those numbers $|f(k,l)|$ which are positive in decreasing order: 
$$
|f(k_1,l)|\geq |f(k_2,l)| \geq |f(k_3,l)| \geq \dots,
$$
where $k_1, k_2, \dots$ is a suitable permutation of (a subset of) $\C$.
Again, this can be done since all level sets $\{(k,l) \in \C^2: ~|f(k,l)|>\lambda\}$ with $\lambda>0$ are finite. Then we find $N_l\in \N$, $N_l\geq 2$ satisfying $\sum_{i=1}^{N_l-1} |f(k_i,l)|^2 <2$ and  $\sum_{i=1}^{N_l} |f(k_i,l)|^2 \geq 2$ and we set $(k_i,l)\in \tilde{S_2}$ for $i=1,\dots, N_l$, and also $(k,l) \in \tilde{S_2}$ if $|f(k,l)|=0$. As previously, we observe that
\begin{equation*}
\sum_{k:~(k,l)\in \tilde{S_2}} |f(k,l)|^2
=|f(k_{N_l},l)|^2 + \sum_{i=1}^{N_l-1} |f(k_i,l)|^2
\leq 1+2=3.
\end{equation*}

For $k\in \C$ we denote 
$$
R_k=\max_{l \in \C:~(k,l)\notin \tilde{S_1}} |f(k,l)|.
$$
We note that $R_k$ is understood to be $0$ if the set $\{l\in \C:~(k,l)\notin \tilde{S_1}\}$ is empty. We now rearrange those numbers $R_k$ that are positive in decreasing order:
$$
R_{\tilde{k}_1} \geq R_{\tilde{k}_2} \geq \dots,
$$
where $\tilde{k}_1, \tilde{k}_2, \dots$ is a permutation of (a subset of) $\C$, indexed by the elements of a set $I_k \subseteq \N$ (we either have $I_k = \emptyset$, or $I_k=\{1,2,\dots,N\}$ for some $N\in \N$, or $I_k=\N$). Such a rearrangement is possible since all level sets $\{(k,l) \in \C^2: ~|f(k,l)|>\lambda\}$ with $\lambda>0$ are finite. We claim that $R_{\tilde{k}_i} \lesssim 1/\sqrt{i}$ for $i\in I_k$, up to an absolute multiplicative constant. To verify this, we fix $i\in I_k$ and consider the set 
$$
A=\{(\tilde{k}_j, l): ~j\in \{1,\dots, i\}, ~(\tilde{k}_j,l)\in \tilde{S_1}, ~f(\tilde{k}_j,l) \neq 0\}.
$$
Then $|f(k,l)|\geq R_{\tilde{k}_i}$ for every $(k,l)\in A$. 
Using the definition of the set $\tilde{S_1}$, we obtain
$$
\sum_{(k,l)\in A} |f(k,l)|^2 = \sum_{j=1}^i \sum_{l:~(\tilde{k}_j,l) \in \tilde{S_1}} |f(\tilde{k}_j,l)|^2 \geq 2i.
$$
Also, recalling that
$$
\sup_{t>0} t^{\frac{1}{4}} f^*(t) =\|f\|_{\ell^{4,\infty}(\C^2)} =1,
$$ 
where $f^*$ denotes the non-increasing rearrangement of $f$ with respect to the counting measure on $\C^2$, we get
$$
\sum_{(k,l)\in A} |f(k,l)|^2 
\leq \sum_{j=1}^{\card{A}} (f^*(j))^2
\leq \sum_{j=1}^{\card{A}} \frac{1}{\sqrt{j}} \approx \sqrt{\card{A}}.
$$
So, $\card{A} \geq c i^2$, where $c>0$ is an absolute constant, which yields 
$$
R_{\tilde{k}_i} \leq f^*(\card{A}) \leq f^*(ci^2) \lesssim \frac{1}{\sqrt{i}},
$$ 
as desired.

Similarly, for $l\in \C$ we denote 
$$
C_l=\max_{k:~(k,l)\notin \tilde{S_2}} |f(k,l)|
$$
and rearrange those numbers $C_l$ that are positive in decreasing order:
$$
C_{\tilde{l}_1} \geq C_{\tilde{l}_2} \geq \dots,
$$
where $\tilde{l}_1, \tilde{l}_2, \dots$ is a permutation of (a subset of) $\C$, indexed by the elements of a set $J_l\subseteq \N$. An argument as above shows that $C_{\tilde{l}_i} \lesssim 1/\sqrt{i}$ for $i\in J_l$. 

We are now ready to define the sets $S_1$ and $S_2$. We notice that if $(k,l) \notin \tilde{S_1}\cup \tilde{S_2}$ then $R_k>0$ and $C_l>0$, and therefore $k=\tilde{k}_i$ for some $i \in \N$ and $l=\tilde{l}_j$ for some $j\in \N$. Thus, we can set 
$$
S_1=\tilde{S_1} \cup \{(k,l)\notin \tilde{S_1}\cup \tilde{S_2}: ~(k,l)=(\tilde{k}_i, \tilde{l}_j) \text{ for $i\geq j$}\}
$$
and 
$$
S_2=\C^2 \setminus S_1=(\tilde{S_2} \setminus \tilde{S_1}) \cup \{(k,l)\notin \tilde{S_1}\cup \tilde{S_2}: ~(k,l)=(\tilde{k}_i, \tilde{l}_j) \text{ for $i< j$}\}.
$$

It remains to verify inequalities~\eqref{E:1} and~\eqref{E:2}. To show~\eqref{E:1}, we fix $k\in \C$ such that $(k,l) \in S_1 \setminus \tilde{S_1}$ for some $l\in \C$. 
Then $k=\tilde{k}_i$ for some $i \in \N$. Now, if $(k,l)=(\tilde{k}_i, \tilde{l}_j) \in S_1 \setminus \tilde{S_1}$ then $j\leq i$, which  means that $\card\{l\in \C: ~(k,l)\in S_1 \setminus \tilde{S_1}\} \leq i$. We also have
$$
|f(k,l)|\leq R_{\tilde{k_i}} \lesssim \frac{1}{\sqrt{i}} \quad \text{if } (k,l) \in S_1\setminus \tilde{S_1}. 
$$
Thus,
\begin{equation}\label{E:tilde}
\sum_{l:~(k,l)\in S_1 \setminus \tilde{S_1}} |f(k,l)|^2 \lesssim i\cdot \left(\frac{1}{\sqrt{i}}\right)^2 =1.
\end{equation} 
A combination of~\eqref{E:S_1} and~\eqref{E:tilde} yields~\eqref{E:1}. Inequality~\eqref{E:2} can be proved analogously.
\end{proof}

\begin{example}
In this example we show that the assumption $f\in \ell^{4,\infty}(\C^2)$ of Lemma~\ref{L:decomposition} is sharp. For simplicity, we work in the setting where $\C=\Z$ but an analogous argument can be applied to cover the general situation as well.
 
Assume that $f$ is a function on $\Z^2$ satisfying the monotonicity assumption 
\begin{equation}\label{E:monotonicity}
|f(k_1,l_1)| \leq |f(k_2,l_2)| \quad \text{if } \max\{|k_1|,|l_1|\} > \max\{|k_2|,|l_2|\}.
\end{equation}
We show that if there are two disjoint subsets $S_1$ and $S_2$ of $\Z^2$ whose union is the entire $\Z^2$ and if there is a constant $C$ for which
\begin{equation*}
\sum_{l:~(k,l)\in S_1} |f(k,l)|^2 \leq C \quad \text{for every $k\in \Z$}
\end{equation*}
and
\begin{equation*}
\sum_{k:~(k,l)\in S_2} |f(k,l)|^2 \leq C  \quad \text{for every $l\in \Z$},
\end{equation*}
then $f\in \ell^{4,\infty}(\Z^2)$. 

Let us fix $M\in \Z^+_0$. Then we have
\begin{align*}
\sum_{k=-M}^{M} \sum_{l=-M}^M |f(k,l)|^2 
&\leq \sum_{k=-M}^M \sum_{l:~(k,l)\in S_1} |f(k,l)|^2
+\sum_{l=-M}^M \sum_{k:~(k,l)\in S_2} |f(k,l)|^2 \\
&\leq 2C(2M+1).
\end{align*}
Due to the monotonicity assumption~\eqref{E:monotonicity}, $f^*((2M+1)^2) \leq |f(k,l)|$ whenever $\max\{|k|, |l|\} \leq M$. This, combined with the previous inequality, yields
$$
(2M+1)^2 (f^*((2M+1)^2))^2 \leq 2C(2M+1),
$$
or
$$
\sup_{M\in \Z^+_0} (2M+1)^{\frac{1}{2}} f^*((2M+1)^2) <\infty.
$$
Thanks to the monotonicity of the function $f^*$, this implies
$$
\sup_{t\in [1,\infty)} t^{\frac{1}{4}} f^*(t)<\infty.
$$
Since $f$ is bounded, we also trivially have 
$$
\sup_{t\in (0,1)} t^{\frac{1}{4}} f^*(t)<\infty,
$$
which yields that $f\in \ell^{4,\infty}(\Z^2)$, as desired.

Finally, let us mention that if a  function $f\in \ell^{4,\infty}(\Z^2)$ satisfies~\eqref{E:monotonicity} then an example of a decomposition of $\Z^2$ having the properties as in Lemma~\ref{L:decomposition} is
$$
S_1=\{(k,l) \in \Z^2: ~|k| \geq |l|\} \quad \text{and} \quad S_2=\{(k,l) \in \Z^2: ~|k|<|l|\}.
$$
\end{example}

\section{Proof of Theorem~\ref{T:sobolev}}\label{S:multiplier_theorem}


In this section we apply Lemma~\ref{L:decomposition} to prove Theorem~\ref{T:sobolev}. We start by introducing some relevant function spaces. Given $1<r<\infty$ and $s>0$, the fractional Sobolev space $L^r_s(\R^{2n})$ is the space of all functions $f$ on $\R^{2n}$ for which
$$
\|f\|_{L^r_s(\R^{2n})}=\|(I-\Delta)^{\frac{s}{2}}f\|_{L^r(\R^{2n})} <\infty. 
$$
Further, the fractional Lorentz-Sobolev space $L^{4,\infty}_s(\R^{2n})$, where $s>0$, is defined as the collection of all functions $f$ on $\R^{2n}$ for which
$$
\|f\|_{L^{4,\infty}_s(\R^{2n})} = \|(I-\Delta)^{\frac{s}{2}} f\|_{L^{4,\infty}(\R^{2n})} <\infty.
$$
It is worth noticing that a function $m$ on $\R^{2n}$ satisfies the assumption~\eqref{E:C_s(m)} if and only if $m \in L^{4,\infty}_s(\R^{2n})$, and the constant $C_s(m)$ in~\eqref{E:C_s(m)} is equal to $\|m\|_{L^{4,\infty}_s(\R^{2n})}$. In addition, since the space $L^4_s(\R^{2n})$ is continuously embedded into $L^{4,\infty}_s(\R^{2n})$, Theorem~\ref{T:sobolev} yields as a corollary that the operator $T_m$ is $L^2(\R^n) \times L^2(\R^n) \to L^1(\R^n)$ bounded whenever $m\in L^4_s(\R^{2n})$. 

We prove Theorem~\ref{T:sobolev} as a consequence of a stronger result involving a space of Besov type. Let us now introduce this function space as well. We let $\varphi_0$ be a smooth function on $\R^{2n}$ satisfying $\varphi_0(\xi)=1$ if $|\xi|\leq 1$ and $\varphi_0(\xi)=0$ if $|\xi|\geq \frac{3}{2}$. Further, if $k\in \N$ then we set $\varphi_k(\xi)=\varphi_0(2^{-k}\xi) - \varphi_0(2^{1-k}\xi)$. 
We denote by $\{\Lambda_k\}_{k=0}^\infty$ the inhomogeneous Littlewood-Paley decomposition, defined via the Fourier transform as $\wh{\Lambda_k f} = \varphi_k \wh{f}$ for $k\in \Z^+_0$. We then let $B(\R^{2n})$ be the space of all functions $f$ on $\R^{2n}$ satisfying
$$
\|f\|_{B(\R^{2n})} = \sum_{k=0}^\infty 2^{\frac{nk}{2}} \|\Lambda_k f\|_{L^{4,\infty}(\R^{2n})}<\infty.
$$
The space $B(\R^{2n})$ is the Besov space with smoothness index $n/2$ and summability index $1$ built upon the Lorentz space $L^{4,\infty}(\R^{2n})$. The family of Besov spaces built upon Lorentz spaces and its relationship to Lorentz-Sobolev spaces was studied in the recent paper~\cite{ST} where it was shown, in particular, that $L^{4,\infty}_s(\R^{2n}) \hookrightarrow B(\R^{2n})$ whenever $s>n/2$. 

\begin{thm}\label{T:besov}
Let $m$ be a function on $\R^{2n}$ belonging to the Besov space $B(\R^{2n})$. Then the associated operator $T_m$ admits a bounded extension from $L^2(\R^n) \times L^2(\R^n)$ to $L^1(\R^n)$ and
$$
\|T_m(f,g)\|_{L^1(\R^n)} \leq C \|m\|_{B(\R^{2n})} \|f\|_{L^2(\R^n)} \|g\|_{L^2(\R^n)}
$$
for some dimensional constant $C$.
\end{thm} 

The following proposition plays a crucial role in the proofs of Theorems~\ref{T:sobolev} and~\ref{T:besov}.

\begin{prop}\label{P:fourier_compact}
Let $k\in \N$ and let $m \in L^{4,\infty}(\R^{2n})$ be a function whose Fourier transform is supported in the ball $\{x\in \R^{2n}: |x|<2^{k}\}$. Then there is a dimensional constant $C$ such that
\begin{equation}\label{E:compact}
\|T_m(f,g)\|_{L^1(\R^n)} \leq C 2^{\frac{nk}{2}} \|m\|_{L^{4,\infty}(\R^{2n})} \|f\|_{L^2(\R^n)} \|g\|_{L^2(\R^n)}.
\end{equation}
\end{prop}

A variant of Proposition~\ref{P:fourier_compact} appeared in~\cite[Proposition 4.1]{KMT}. On the one hand, \cite[Proposition 4.1]{KMT} is more general as it applies to the setting of pseudodifferential operators with symbols whose Fourier transform is supported in a rectangle centered at the origin. On the other hand, once we restrict ourselves to the special setting of Proposition~\ref{P:fourier_compact} then Proposition~\ref{P:fourier_compact} provides a more precise estimate than~\cite[Proposition 4.1]{KMT}, in the sense that it lowers the power of $2$ on the right-hand side of~\eqref{E:compact} from $2^{nk}$ to $2^{nk/2}$. 

We will need several auxiliary results to prove Proposition~\ref{P:fourier_compact}. We start with the following straightforward lemma.

\begin{lm}\label{L:schwartz}
Let $\Phi$ be a non-negative function on $\R^n$ satisfying
$$
\sup_{x\in \R^n} \Phi(x) (1+|x|)^\gamma <\infty
$$ 
for some $\gamma >n$. Then
$$
\sup_{x\in \R^n} \sum_{k\in \Z^n} \Phi(x-k)<\infty.
$$
\end{lm}

\begin{proof}
Given $x=(x_1, x_2, \dots, x_n)\in \R^n$, we set 
$$
\lfloor x \rfloor=(\lfloor x_1 \rfloor, \lfloor x_2 \rfloor, \dots, \lfloor x_n \rfloor) \in \Z^n.
$$ 
Then $|x-\lfloor x \rfloor|<\sqrt{n} \leq n$, which yields that for any $k\in \Z^n$,
\begin{align*}
1+|x-k| 
&\geq 1+\left||x-\lfloor x \rfloor|-|\lfloor x \rfloor-k|\right| \\
&\geq 
\begin{cases}
1+\frac{|\lfloor x \rfloor-k|}{2} & \text{if } |\lfloor x \rfloor-k|\geq 2n\\
1 & \text{if }|\lfloor x \rfloor-k|<2n\\
\end{cases}\\
&\geq \frac{1}{2n+1}(1+|\lfloor x \rfloor-k|).
\end{align*}
Thus,
\begin{align*}
\sup_{x\in \R^n} \sum_{k\in \Z^n} \Phi(x-k) 
&\lesssim \sup_{x\in \R^n} \sum_{k\in \Z^n} (1+|x-k|)^{-\gamma}\\
&\lesssim \sup_{x\in \R^n} \sum_{k\in \Z^n} (1+|k-\lfloor x \rfloor|)^{-\gamma}\\
&=\sum_{k\in \Z^n} (1+|k|)^{-\gamma} <\infty. 
\end{align*}
\end{proof}

A key tool for proving Proposition~\ref{P:fourier_compact} is a representation of functions in terms of product-type wavelets. In particular, we will make use of the following fact, due to Meyer~\cite{M1,M2}:

There exist real-valued functions $\Psi_F$, $\Psi_M \in \mathcal S(\R)$ such that $\wh{\Psi_F}$ is compactly supported, $\wh{\Psi_M}$ is compactly supported away from the origin and, if we denote
\begin{align*}
\Psi_\beta(x)=\prod_{r=1}^{2n} \Psi_F(x_r-\beta_r), \quad &\beta = (\beta_1, \dots, \beta_{2n}) \in \Z^{2n},\\ 
&x=(x_1, \dots, x_{2n})\in \R^{2n},
\end{align*}
and
$$
\Psi_\beta^G(x)=\prod_{r=1}^{2n} \Psi_{G_r}(x_r-\beta_r), \quad \beta\in \Z^{2n}, ~x\in \R^{2n}, 
$$
where $G=(G_1,\dots, G_{2n}) \in \{F,M\}^{2n}\setminus \{(F, \dots, F)\}$,
then the family of functions 
$$
\Psi_\beta^{j,G}(x)=
\begin{cases}
\Psi_\beta(x), & j=0, ~ G=(F, \dots, F)\\
2^{(j-1)n} \Psi_\beta^G(2^{j-1}x), & j\in \mathbb N, ~ G\in \{F,M\}^{2n}\setminus \{(F, \dots, F)\}
\end{cases}
$$
forms an orthonormal basis in $L^2(\R^{2n})$. 

In addition, the same family of functions is also an unconditional basis in any Lebesgue space $L^p(\R^{2n})$ with $1<p<\infty$. Thus, setting
\begin{align*}
J
=&\{(j,G): ~j=0 \text{ and } G=(F, \dots, F),\\ 
&\text{ or } j\in \N \text{ and } G\in \{F,M\}^{2n} \setminus \{(F,\dots, F)\}\},
\end{align*}
any function $f$ in $L^p(\R^{2n})$ can be expressed in the form
$$
f=\sum_{(j,G)\in J} \sum_{\beta\in \Z^{2n}} \langle f,\Psi_\beta^{j,G}\rangle \Psi_\beta^{j,G},
$$
unconditional convergence being in $\mathcal S'(\R^{2n})$. Also, we have the equivalence
\begin{equation}\label{E:triebel}
\|f\|_{L^{p}(\R^{2n})} \approx \left\|\left(\sum_{(j,G)\in J} \sum_{\beta \in \Z^{2n}} |\langle f,\Psi_\beta^{j,G}\rangle 2^{jn} \chi_{j\beta} |^2 \right)^{\frac{1}{2}}\right\|_{L^{p}(\R^{2n})},
\end{equation} 
where $\chi_{j \beta}$ denotes the characteristic function of the cube centered at $2^{-j}\beta$ with side-length $2^{1-j}$. For the proof of this statement, see, e.g.,~\cite[Theorem 3.12]{TriebelTFSIII}.

\medskip
The final ingredient needed for the proof of Proposition~\ref{P:fourier_compact} is the following lemma, which can be obtained from~\eqref{E:triebel} using real interpolation.

\begin{lm}\label{L:discrete}
Suppose that $m\in L^{4,\infty}(\R^{2n})$. Let $(j,G)\in J$ and let $a^{j,G}=\{a^{j,G}_\beta\}_{\beta \in \Z^{2n}}$ be the sequence defined by $a^{j,G}_\beta=\langle m,\Psi_\beta^{j,G}\rangle$ for $\beta \in \Z^{2n}$. Then 
$$
\|a^{j,G}\|_{\ell^{4,\infty}(\Z^{2n})} \lesssim 2^{-\frac{jn}{2}} \|m\|_{L^{4,\infty}(\R^{2n})}.
$$
\end{lm}

\begin{proof}
Let $S$ be the sublinear operator defined as
$$
Sf =\left(\sum_{(j,G)\in J} \sum_{\beta\in \Z^{2n}} |\langle f,\Psi_\beta^{j,G}\rangle 2^{jn} \widetilde{\chi_{j\beta}} |^2 \right)^{\frac{1}{2}},
$$
where $\widetilde{\chi_{j \beta}}$ denotes the characteristic function of the cube centered at $2^{-j}\beta$ with side-length $2^{-j}$. Since $\widetilde{\chi_{j \beta}} \leq \chi_{j \beta}$, it follows from~\eqref{E:triebel} that the operator $S$ is bounded on $L^p(\R^{2n})$ for every $p\in (1,\infty)$. The Marcinkiewicz interpolation theorem (see, e.g.,~\cite[Chapter 4, Theorem 4.13]{BS}) then yields that $S$ is bounded on $L^{4,\infty}(\R^{2n})$, that is,
$$
\left\|\left(\sum_{(j,G)\in J} \sum_{\beta \in \Z^{2n}} |\langle m,\Psi_\beta^{j,G}\rangle 2^{jn} \widetilde{\chi_{j\beta}} |^2 \right)^{\frac{1}{2}}\right\|_{L^{4,\infty}(\R^{2n})} \lesssim \|m\|_{L^{4,\infty}(\R^{2n})}. 
$$
Therefore, fixing a pair $(j,G)$ and using that the supports of the functions $\widetilde{\chi_{j \beta}}$ do not overlap in $\beta$ (except perhaps for the boundary of the corresponding cubes), we obtain
\begin{equation}\label{E:continuous}
2^{jn} \left\|\sum_{\beta \in \Z^{2n}} a^{j,G}_\beta \widetilde{\chi_{j\beta}}\right\|_{L^{4,\infty}(\R^{2n})} \lesssim \|m\|_{L^{4,\infty}(\R^{2n})}. 
\end{equation}
Now, for any $\lambda>0$,
\begin{align*}
\left|\left\{x\in \R^{2n}:~ \sum_{\beta \in \Z^{2n}} |a^{j,G}_\beta| \widetilde{\chi_{j\beta}}(x) >\lambda\right\}\right|&\\
=2^{-2jn} \card\{\beta\in \Z^{2n}:~ |a^{j,G}_\beta|>\lambda\}&,
\end{align*}
which implies that
\begin{align*}
\left\|\sum_{\beta \in \Z^{2n}} a^{j,G}_\beta \widetilde{\chi_{j\beta}}\right\|_{L^{4,\infty}(\R^{2n})}
&=\sup_{\lambda>0} \lambda \left|\left\{x\in \R^{2n}:~ \sum_{\beta \in \Z^{2n}} |a^{j,G}_\beta| \widetilde{\chi_{j\beta}}(x) >\lambda\right\}\right|^{\frac{1}{4}}\\
&=2^{-\frac{jn}{2}}\sup_{\lambda>0} \lambda \left(\card\{\beta\in \Z^{2n}:~ |a^{j,G}_\beta|>\lambda\}\right)^{\frac{1}{4}}\\
&=2^{-\frac{jn}{2}} \|a^{j,G}\|_{\ell^{4,\infty}(\Z^{2n})}.
\end{align*}
Combining this equality with~\eqref{E:continuous} yields the conclusion.
\end{proof}

\begin{proof}[Proof of Proposition~\ref{P:fourier_compact}]
Let $(j,G)\in J$ and let $a^{j,G}$ be the sequence defined in Lemma~\ref{L:discrete}.
We observe that 
$$
a^{j,G}_\beta
=\left<m, \Psi_\beta^{j,G} \right>
=\left< \wh{m}, (\Psi_\beta^{j,G})\spcheck \right>
$$
and that for every $j\in \N$, the inverse Fourier transform of $\Psi_\beta^{j,G}$ is supported in the annulus $\{x\in \R^{2n}:~K_1 2^j <|x|<K_2 2^j\}$, where $K_1$ and $K_2$ are suitable dimensional constants. Using the support properties of $\wh{m}$, we thus deduce that $a^{j,G}_\beta=0$ whenever $j\in \N$ satisfies $K_1 2^j\geq 2^k$. 

Since $m \in L^{4,\infty}(\R^{2n})$, we can write $m=m_1+m_2$, where $m_1 \in L^3(\R^{2n})$ and $m_2 \in L^5(\R^{2n})$. As the family $\{\Psi_\beta^{j,G}\}$ is an unconditional basis in both $L^3(\R^{2n})$ and $L^5(\R^{2n})$, we deduce that $m$ can be represented as
$$
m=\sum_{(j,G)\in J} \sum_{\beta\in \Z^{2n}} a^{j,G}_\beta \Psi_\beta^{j,G},
$$
unconditional convergence being in $\mathcal S'(\R^{2n})$. Consequently, whenever $f$ and $g$ are Schwartz functions on $\R^{n}$, we have the pointwise identity
\begin{align*}
T_m(f,g)(x)
&= \sum_{(j,G)\in J} \sum_{\beta\in \Z^{2n}} a^{j,G}_\beta T_{\Psi_\beta^{j,G}} (f,g)(x)\\
&=\sum_{\substack{(j,G)\in J\\ j=0 \text{ or } K_1 2^j \leq 2^k}} \sum_{\beta\in \Z^{2n}} a^{j,G}_\beta T_{\Psi_\beta^{j,G}} (f,g)(x),\quad x\in \R^n.
\end{align*}

Let $U$ be any subset of $\{(j,G)\in J: ~j=0 \text{ or } K_1 2^j\leq 2^k\}$ and let $V$ be any finite subset of $\Z^{2n}$. By the Fatou lemma,
\begin{align*}
\|T_m(f,g)\|_{L^1(\R^n)} 
&\leq \sup_{U, V}  \left\|\sum_{(j,G)\in U} \sum_{\beta \in V} a^{j,G}_\beta T_{\Psi_\beta^{j,G}} (f,g)\right\|_{L^1(\R^n)}\\
&\leq \sup_{U, V} \sum_{(j,G)\in U} \left\|\sum_{\beta \in V} a^{j,G}_\beta T_{\Psi_\beta^{j,G}} (f,g)\right\|_{L^1(\R^n)}.
\end{align*}

We fix $(j,G)\in U$. We shall use the notation $\beta=(k,l)$, $(k,l)\in \Z^n \times \Z^n=\Z^{2n}$, 
$$
a_{\beta}^{j,G}=a_{k,l} \quad \text{and} \quad \Psi_{\beta}^{j,G}(\xi,\eta)=\omega_{1,k}(\xi) \omega_{2,l}(\eta),
$$ 
where
$$
\omega_{1,k}(\xi)=
\begin{cases}
\prod_{r=1}^n \Psi_F(\xi_r-k_r), & j=0,\\ 
2^{\frac{(j-1)n}{2}} \prod_{r=1}^n \Psi_{G_r}(2^{j-1}\xi_r-k_r), &j\in \N,
\end{cases}
$$
and
$$
\omega_{2,l}(\eta)=
\begin{cases}
\prod_{r=1}^n \Psi_F(\eta_r-l_r), & j=0,\\ 
2^{\frac{(j-1)n}{2}} \prod_{r=1}^n \Psi_{G_{r+n}}(2^{j-1}\eta_r-l_r), &j\in \N.
\end{cases}
$$

According to Lemma~\ref{L:decomposition}, there are two pairwise disjoint sets $S_1$, $S_2$ in $\Z^{2n}$ such that $S_1 \cup S_2=\Z^{2n}$,
$$
\left(\sum_{l:~(k,l)\in S_1} |a_{k,l}|^2\right)^{\frac{1}{2}} \leq C \|a\|_{\ell^{4,\infty}(\Z^{2n})} \quad \text{for every $k\in \Z^n$}
$$
and
$$
\left(\sum_{k:~(k,l)\in S_2} |a_{k,l}|^2\right)^{\frac{1}{2}} \leq C \|a\|_{\ell^{4,\infty}(\Z^{2n})} \quad \text{for every $l\in \Z^n$}.
$$

Let us now use Lemma~\ref{L:schwartz} to derive two preliminary estimates that will be needed later on. 
Assume that $j\in \mathbb N$ and $f\in \mathcal S(\R^n)$. Then the function $\Phi(\xi)=\prod_{r=1}^n \Psi^2_{G_r}(\xi_r)$ belongs to the Schwartz space $\mathcal S(\R^n)$, and therefore satisfies the assumption of Lemma~\ref{L:schwartz}. Consequently,
\begin{align}\label{E:omega}
&\sum_{k\in \Z^n} \|\omega_{1,k} \wh{f}\|^2_{L^2(\R^n)}
=\int_{\R^n} |\wh{f}(\xi)|^2 \sum_{k\in \Z^n} |\omega_{1,k}(\xi)|^2\,d\xi\\
&\approx 2^{jn} \int_{\R^n} |\wh{f}(\xi)|^2 \sum_{k\in \Z^n} \Phi(2^{j-1}\xi -k) \,d\xi \nonumber\\
&\lesssim2^{jn} \|f\|_{L^2(\R^n)}^2. \nonumber
\end{align}

To derive the second estimate, we apply Lemma~\ref{L:schwartz} with the second power of the functions 
$$
\Phi_1(\xi)=(1+|\xi|)^{-n} \quad \text{and} \quad \Phi_2(\xi)=(1+|\xi|)^{n} \prod_{r=1}^n\Psi_{G_r}(\xi_r), \quad \xi \in \R^n.
$$ 
We notice that 
\begin{equation*}
\omega_{1,k}(\xi)
 = 2^{\frac{(j-1)n}{2}}\Phi_1(2^{j-1}\xi -k) \Phi_2(2^{j-1}\xi -k).
\end{equation*}
Therefore,
\begin{align}\label{E:omega2}
&\sum_{l \in \Z^n} \big\|\sum_{k: ~ (k,l) \in S_1 \cap V} a_{k,l} \omega_{1,k} \wh{f}\big\|_{L^2(\R^n)}^2\\
&= 2^{(j-1)n}\sum_{l\in \Z^n} \big\|\sum_{k:~(k,l) \in S_1 \cap V} a_{k,l} \Phi_1(2^{j-1}\xi-k) \Phi_2(2^{j-1}\xi-k) \wh{f}(\xi)\big\|_{L^2(\R^n)}^2 \nonumber\\
&\lesssim 2^{jn} \sum_{l\in \Z^n} \int_{\R^n} |\wh{f}(\xi)|^2 \big( \sum_{k:~(k,l) \in S_1 \cap V} |a_{k,l}|^2 \Phi_1^2(2^{j-1}\xi-k) \big) \big( \sum_{k\in \Z^n} \Phi_2^2(2^{j-1}\xi-k) \big) \,d\xi \nonumber\\
&\lesssim 2^{jn} \sum_{k\in \Z^n} \sum_{l:~(k,l) \in S_1 \cap V} |a_{k,l}|^2 \int_{\R^n} |\wh{f}(\xi)|^2 \Phi_1^2(2^{j-1}\xi-k)\,d\xi \nonumber\\
&\lesssim 2^{jn} \|a\|_{\ell^{4,\infty}(\Z^{2n})}^2 \int_{\R^n} |\wh{f}(\xi)|^2 \sum_{k\in \Z^n} \Phi_1^2(2^{j-1}\xi-k)\,d\xi \nonumber \\
&\lesssim 2^{jn} \|a\|_{\ell^{4,\infty}(\Z^{2n})}^2 \|f\|_{L^2(\R^n)}^2. \nonumber
\end{align}
One also has estimates analogous to~\eqref{E:omega} and~\eqref{E:omega2} if $j=0$ or if $\omega_{1,k}$ is replaced by $\omega_{2,l}$, $f$ is replaced by $g$ and $S_1$ is replaced by $S_2$.

Applying the inequalities we have just derived, we conclude that
\begin{align*}
&\big\|\sum_{\beta \in V} a^{j,G}_\beta T_{\Psi_\beta^{j,G}}(f,g)\big\|_{L^1(\R^n)}\\
&\leq \big\|\sum_{(k,l)\in S_1 \cap V} a_{k,l} \mathcal F^{-1}(\omega_{1,k} \wh{f}) \mathcal F^{-1}(\omega_{2,l} \wh{g})\big\|_{L^1(\R^n)}\\
&+\big\|\sum_{(k,l)\in S_2 \cap V} a_{k,l} \mathcal F^{-1}(\omega_{1,k} \wh{f}) \mathcal F^{-1}(\omega_{2,l} \wh{g})\big\|_{L^1(\R^n)}\\
&\leq \sum_{l\in \Z^n} \big\|\sum_{k: ~(k,l)\in S_1 \cap V} a_{k,l} \mathcal F^{-1}(\omega_{1,k} \wh{f}) \mathcal F^{-1}(\omega_{2,l} \wh{g})\big\|_{L^1(\R^n)}\\
&+ \sum_{k\in \Z^n} \big\|\sum_{l: ~(k,l)\in S_2 \cap V} a_{k,l} \mathcal F^{-1}(\omega_{1,k} \wh{f}) \mathcal F^{-1}(\omega_{2,l} \wh{g})\big\|_{L^1(\R^n)}\\
&\leq \sum_{l\in \Z^n} \|\omega_{2,l} \wh{g}\|_{L^2(\R^n)}  \big\|\sum_{k: ~(k,l)\in S_1 \cap V} a_{k,l} \omega_{1,k} \wh{f}\big\|_{L^2(\R^n)}\\
&+ \sum_{k\in \Z^n} \|\omega_{1,k} \wh{f}\|_{L^2(\R^n)}  \big\|\sum_{l: ~(k,l)\in S_2 \cap V} a_{k,l} \omega_{2,l} \wh{g}\big\|_{L^2(\R^n)}\\
&\lesssim \left( \sum_{l\in \Z^n} \|\omega_{2,l} \wh{g}\|^2_{L^2(\R^n)} \right)^{\frac{1}{2}} \left(\sum_{l\in \Z^n} \big\|\sum_{k: ~(k,l)\in S_1 \cap V} a_{k,l} \omega_{1,k} \wh{f}\big\|^2_{L^2(\R^n)} \right)^{\frac{1}{2}}\\
&+ \left( \sum_{k\in \Z^n} \|\omega_{1,k} \wh{f}\|^2_{L^2(\R^n)} \right)^{\frac{1}{2}} \left(\sum_{k\in \Z^n} \big\|\sum_{l: ~(k,l)\in S_2 \cap V} a_{k,l} \omega_{2,l} \wh{g}\big\|^2_{L^2(\R^n)} \right)^{\frac{1}{2}}\\
&\lesssim 2^{jn} \|a\|_{\ell^{4,\infty}(\Z^{2n})} \|f\|_{L^2(\R^n)} \|g\|_{L^2(\R^n)}.
\end{align*}
Using this and Lemma~\ref{L:discrete}, we obtain 
\begin{align*}
\|T_m(f,g)\|_{L^1(\R^n)} 
&\leq \sup_{U} \left(\sum_{(j,G)\in U} 2^{jn} \|a^{j,G}\|_{\ell^{4,\infty}(\Z^{2n})} \right)\|f\|_{L^2(\R^n)} \|g\|_{L^2(\R^n)}\\
&\lesssim \left(\sum_{\substack{(j,G)\in J\\ j=0 \text{ or } K_1 2^j\leq 2^k}} 2^{\frac{jn}{2}} \right) \|m\|_{L^{4,\infty}(\R^{2n})} \|f\|_{L^2(\R^n)} \|g\|_{L^2(\R^n)}\\
&\approx 2^{\frac{kn}{2}}\|m\|_{L^{4,\infty}(\R^{2n})} \|f\|_{L^2(\R^n)} \|g\|_{L^2(\R^n)}.
\end{align*}
\end{proof}

\begin{proof}[Proof of Theorem~\ref{T:besov}]
We observe that the Littlewood-Paley decomposition $\{\Lambda_k m\}_{k=0}^\infty$ of $m$ has the property that $\wh{\Lambda_k m}$ is supported in the ball $\{x\in \R^{2n}:~ |x|<2^{k+2}\}$ for each $k\in \Z^+_0$. Proposition~\ref{P:fourier_compact} then yields
\begin{align*}
\|T_m(f,g)\|_{L^1(\R^n)}
&\leq \sum_{k=0}^\infty \|T_{\Lambda_k m} (f,g)\|_{L^1(\R^n)}\\
&\lesssim \sum_{k=0}^\infty 2^{\frac{nk}{2}} \|\Lambda_k m\|_{L^{4,\infty}(\R^{2n})} \|f\|_{L^2(\R^n)} \|g\|_{L^2(\R^n)}\\
&=\|m\|_{B(\R^{2n})} \|f\|_{L^2(\R^n)} \|g\|_{L^2(\R^n)}.
\end{align*}
\end{proof} 

\begin{proof}[Proof of Theorem~\ref{T:sobolev}]
This is a consequence of Theorem~\ref{T:besov} and of the embedding $L^{4,\infty}_s(\R^{2n}) \hookrightarrow B(\R^{2n})$ which holds whenever $s>\frac{n}{2}$ according to~\cite[Theorem 1.2]{ST}. 
\end{proof}

\section{Proof of Theorem~\ref{T:sharpness}}\label{S:example}

In this section we focus on proving the sharpness of Theorem~\ref{T:sobolev}, as stated in Theorem~\ref{T:sharpness}. The following proposition is an important step towards achieving this goal.

\begin{prop}\label{P:bumps}
Assume that $\Psi$ is a nontrivial smooth function supported in the set $\{x\in \R^{2n}: ~ |x|_{\infty}<\frac{1}{10}\}$. Given a sequence $c=\{c_{k,l}\}_{(k,l)\in \Z^n \times \Z^n}$ of complex numbers, consider the function $m$ defined by
\begin{equation}\label{E:form}
m(\xi,\eta)=\sum_{(k,l)\in \Z^n \times \Z^n} c_{k,l} \Psi(\xi-k,\eta-l), \quad (\xi,\eta)\in \R^{n}\times \R^n.
\end{equation}
If $c\in \ell^{4,\infty}(\Z^{2n})$ then the operator $T_m$ admits a bounded extension from $L^2(\R^n) \times L^2(\R^n)$ to $L^1(\R^n)$, and there is a constant $C$ depending on $n$ and $\Psi$ such that
$$
\|T_m(f,g)\|_{L^1(\R^n)} \leq C \|c\|_{\ell^{4,\infty}(\Z^{2n})} \|f\|_{L^2(\R^n)} \|g\|_{L^2(\R^n)}.
$$

Conversely, assume that $d=\{d_{k,l}\}_{(k,l)\in \Z^n \times \Z^n}$ is a bounded sequence of non-negative numbers which satisfies the monotonicity assumption
\begin{equation}\label{E:monotonicity_c}
d_{k_1,l_1}  \leq d_{k_2,l_2} \quad \text{if } |(k_1,l_1)|_\infty > |(k_2,l_2)|_\infty
\end{equation}
and does not belong to $\ell^{4,\infty}(\Z^{2n})$.
Then there exists a sequence $c=\{c_{k,l}\}_{(k,l)\in \Z^n \times \Z^n}$ of real numbers such that $|c_{k,l}|=d_{k,l}$ for all $(k,l)\in \Z^{n} \times \Z^n$, for which the associated operator $T_m$ is unbounded from $L^2(\R^n) \times L^2(\R^n)$ to $L^1(\R^n)$.
\end{prop}


\begin{proof}
We first verify that if $c\in \ell^{4,\infty}(\Z^{2n})$ then the function $m$ given by~\eqref{E:form} belongs to the Lorentz-Sobolev space $L^{4,\infty}_s(\R^{2n})$ for some $s>\frac{n}{2}$. To this end, we fix an even integer $s>\frac{n}{2}$ and observe that
\begin{align*}
\|m\|_{L^{4,\infty}_s(\R^{2n})}
&=\|(I-\Delta)^{\frac{s}{2}}m\|_{L^{4,\infty}(\R^{2n})}\\
&\lesssim \sup_{|\alpha|\leq s} \|\partial^\alpha m\|_{L^{4,\infty}(\R^{2n})}
\lesssim \|c\|_{\ell^{4,\infty}(\Z^{2n})}.
\end{align*}
Applying Theorem~\ref{T:sobolev}, we obtain the first part of the proposition. 

Let us now focus on the second part of the proposition. The assumption $d\notin \ell^{4,\infty}(\Z^{2n})$ tells us that 
$$
\sup_{t>0} t^{\frac{1}{4}} d^*(t)=\infty.
$$
Since $d$ is a bounded sequence, we have
$$
\sup_{t\in (0,4^{2n})} t^{\frac{1}{4}} d^*(t)<\infty,
$$
and thus
$$
\sup_{t\in [4^{2n},\infty)} t^{\frac{1}{4}} d^*(t)=\infty.
$$
Furthermore, thanks to the monotonicity of the function $d^*$, the last equality implies
$$
\sup_{N\in \N} N^{\frac{n}{2}} d^*((4N)^{2n}) =\infty. 
$$
Therefore, we can find an increasing sequence $\{b_K\}_{K\in \N}$ of positive integers satisfying $b_{K+1} >2b_K$ for every $K\in \N$ and
$$
\lim_{K\to \infty} b_K^{\frac{n}{2}} d^*((4b_K)^{2n}) =\infty.
$$
We set $\rho_K=(4b_K)^{2n}$, $K\in \N$. Then 
\begin{equation}\label{E:lim}
\lim_{K\to \infty} \rho_K^{\frac{1}{4}} d^*(\rho_K) = \infty.
\end{equation}

We denote by $\{a_{l}(t)\}_{l\in \Z^n}$ the sequence of Rademacher functions indexed by the elements of the countable set $\Z^n$. For any given $t\in [0,1]$ we define the function
$$
m_t(\xi,\eta)=\sum_{(j,k)\in \Z^{2n}} a_{j+k}(t) d_{j,k} \Psi(\xi-j, \eta-k), \quad (\xi, \eta) \in \R^n \times \R^n.
$$
Further, let $\varphi$ be a Schwartz function on $\R^n$ whose Fourier transform is supported in the set $\{\xi \in \R^n: ~ |\xi|_\infty<\frac{1}{5}\}$ and which satisfies $\wh{\varphi}(\xi)=1$ if $|\xi|_\infty \leq \frac{1}{10}$. Given any $K\in \N$, we denote $I_K=\{b_K, b_K+1, \dots, 2b_K-1\}$ and define $f_K=g_K$ to be the functions on $\R^n$ whose Fourier transform satisfies
$$
\wh{f_K}(\xi) 
=\sum_{j\in I_K^n} \wh{\varphi}(\xi-j), \quad \xi \in \R^n.
$$
Then
$$
m_t(\xi,\eta) \wh{f_K}(\xi) \wh{g_K}(\eta) 
=\sum_{j\in I_K^n} \sum_{k\in I_K^n} a_{j+k}(t) d_{j,k} \Psi(\xi-j,\eta-k).
$$
This yields
\begin{align*}
T_{m_t}(f_K,g_K)(x)
&=(\mathcal F^{-1}\Psi) (x,x) \sum_{j\in I_K^n} \sum_{k\in I_K^n} a_{j+k}(t) d_{j,k} e^{2\pi i x \cdot (j+k)}\\
&=(\mathcal F^{-1}\Psi) (x,x)\sum_{l\in I_K^n+I_K^n} a_l(t) e^{2\pi i x \cdot l} \sum_{j\in I_K^n: ~l-j \in I_K^n} d_{j,l-j}
\end{align*}
for $x\in \R^n$.
By Fubini's theorem and Khintchine's inequality,
\begin{align}\label{E:rho}
\int_0^1 &\|T_{m_t}(f_K,g_K)\|_{L^1(\R^n)}\,dt
=\int_{\R^n} \int_0^1 |T_{m_t}(f_K,g_K)(x)|\,dt\,dx\\
\nonumber
&\approx \int_{\R^n} |(\mathcal F^{-1}\Psi) (x,x)| \,dx \big(\sum_{l\in I_K^n + I_K^n} \big|\sum_{j\in I_K^n: ~l-j \in I_K^n} d_{j,l-j} \big|^2 \big)^{\frac{1}{2}}\\
\nonumber
&\gtrsim d^*((4b_K)^{2n}) \big(\sum_{l\in I_K^n + I_K^n} \card\{j\in I_K^n: ~l-j \in I_K^n\}^2 \big)^{\frac{1}{2}}\\
\nonumber
&\gtrsim d^*(\rho_K) (b_K)^{\frac{3n}{2}}
\gtrsim \rho_K^{\frac{3}{4}} d^*(\rho_K).
\end{align}
Here, we have used that $d_{j,k}\geq d^*((4b_K)^{2n})$ if $(j,k)\in I_K^{n} \times I_K^n$, a fact which follows from the monotonicity assumption~\eqref{E:monotonicity_c}. Estimate~\eqref{E:rho} now yields that there is a sequence $\{t_K\}_{K\in \N}$ of numbers in $[0,1]$ such that
$$
\|T_{m_{t_K}}(f_K,g_K)\|_{L^1(\R^n)} \gtrsim \rho_K^{\frac{3}{4}} d^*(\rho_K).
$$
Let us define the sequence $c$ as
$$
c_{j,k}=
\begin{cases}
a_{j+k}(t_K) d_{j,k} & \text{if } (j,k)\in I_K^n \times I_K^n \text{ for some } K\in \N\\
d_{j,k} & \text{if } (j,k) \in \Z^{2n}  \setminus \bigcup_{K\in \N} (I_K^n \times I_K^n).
\end{cases}
$$
We notice that this definition is correct since the sets $I_K$ are pairwise disjoint. 

Let $m$ be the function given by~\eqref{E:form}. Using the support properties of $f_K$ and $g_K$, it is not difficult to observe that
$$
T_m(f_K,g_K) = T_{m_{t_K}}(f_K,g_K)
$$
for every $K\in \N$. Consequently,
$$
\|T_{m}(f_K,g_K)\|_{L^1(\R^n)} \gtrsim \rho_K^{\frac{3}{4}} d^*(\rho_K).
$$
Assume, for the sake of contradiction, that $T_m$ is bounded from $L^2(\R^n) \times L^2(\R^n)$ to $L^1(\R^n)$. Then
$$
\rho_K^{\frac{3}{4}} d^*(\rho_K)
\lesssim \|f_K\|_{L^2(\R^n)} \|g_K\|_{L^2(\R^n)}
\lesssim \rho_K^{\frac{1}{2}}.
$$
This yields
$$
\sup_{k\in \N}\rho_K^{\frac{1}{4}} d^*(\rho_K) <\infty,
$$
which contradicts~\eqref{E:lim}. The proof is complete.
\end{proof}

Having Proposition~\ref{P:bumps} at our disposal, we can now prove Theorem~\ref{T:sharpness}.

\begin{proof}[Proof of Theorem~\ref{T:sharpness}]
Assume that $\Psi$ is a function as in Proposition~\ref{P:bumps} which satisfies, in addition, the pointwise estimate $|(I-\Delta)^{\frac{s}{2}}\Psi| \leq 1$. Also, let
$m$ be any function of the form~\eqref{E:form}. Since $s\in 2\N$, we have that all functions of the form $(I-\Delta)^{\frac{s}{2}} [\Psi(\xi-k, \eta-l)]$ for some $(k,l)\in \Z^{2n}$ are compactly supported in a set of measure less than $1$ and their supports are pairwise disjoint in $k$ and $l$. Therefore, 
\begin{align}\label{E:measure}
&|\{x\in \R^{2n}: ~|(I-\Delta)^{\frac{s}{2}}m(x)|>\lambda\}|\\ 
&\leq \card\{(k,l) \in \Z^{2n}: ~|c_{k,l}|>\lambda\}, \quad \lambda \in (0,1). \nonumber
\end{align}

Let $\tau(\lambda)$ be the right-continuous function on $(0,1)$ that is equal to $\lfloor \mu(\lambda) \rfloor$ for a.e.\ $\lambda \in (0,1)$. Since the function $\tau$ is right-continuous, non-increasing and has values in $\Z^+_0$, there is a sequence $d=\{d_{k,l}\}_{(k,l)\in \Z^{2n}}$ of non-negative numbers satisfying the monotonicity assumption~\eqref{E:monotonicity_c} for which 
\begin{equation}\label{E:tau}
\card\{(k,l)\in \Z^{2n}: ~d_{k,l} >\lambda \} = \tau(\lambda), \quad \lambda \in (0,1).
\end{equation}
Since $\sup_{\lambda \in (0,1)} \lambda (\mu(\lambda))^{\frac{1}{4}} =\infty$ and $\mu$ is non-increasing, we necessarily have $\lim_{\lambda \to 0_+} \mu(\lambda)=\infty$. Therefore, $\tau(\lambda)$ is comparable to $\mu(\lambda)$ for all but countably many $\lambda \in (0,\lambda_0)$, with $\lambda_0>0$ sufficiently small. Consequently,
$$
\sup_{\lambda \in (0,\lambda_0)} \lambda (\tau(\lambda))^{\frac{1}{4}}
\gtrsim \sup_{\lambda \in (0,\lambda_0)} \lambda (\mu(\lambda))^{\frac{1}{4}}
=\infty,
$$
since
$$
\sup_{\lambda \in [\lambda_0, 1)} \lambda (\mu(\lambda))^{\frac{1}{4}} \leq (\mu(\lambda_0))^{\frac{1}{4}}<\infty.
$$
This implies that $d$ does not belong to $\ell^{4,\infty}(\Z^{2n})$, and so, owing to Proposition~\ref{P:bumps}, there is a sequence $\{c_{k,l}\}_{(k,l)\in \Z^{2n}}$ of real numbers such that $|c_{k,l}|=d_{k,l}$ for all $(k,l)\in \Z^{2n}$ and for which the associated operator $T_m$, with $m$ defined by~\eqref{E:form}, is unbounded from $L^2(\R^n) \times L^2(\R^n)$ to $L^1(\R^n)$. In addition, using~\eqref{E:measure} and~\eqref{E:tau}, we deduce that
$$
|\{x\in \R^{2n}: ~|(I-\Delta)^{\frac{s}{2}}m(x)|>\lambda\}| 
\leq \tau(\lambda) \leq \mu(\lambda), \quad \lambda \in (0,1),
$$
as desired.
\end{proof}

\section{Proof of Theorem~\ref{T:q=4}}\label{S:q=4}

In this section we prove Theorem~\ref{T:q=4}. To this end, 
for every $N \in 2\N$ we consider the ``interval'' 
$$
I_N=\{b_N, b_N+1,\dots, b_N+2^{N^2+N/2}-1\},
$$ 
where $\{b_N\}_{N\in 2\N}$ is an increasing sequence of integers to be specified later. 
We also set
$$
S_N=I_N^n \subseteq \Z^n.
$$

\begin{lm}\label{L:S_N}
We have
$$
\sum_{l\in S_N +S_N} \left(\operatorname{card}\{j\in S_N: ~l-j \in S_N\}\right)^2 \gtrsim 2^{3n(N^2+\frac{N}{2})}.
$$
\end{lm}

\begin{proof}
We set
$$
J_N=I_N+I_N=\{2b_N, 2b_N+1,\dots, 2b_N+2^{N^2+N/2+1}-2\}
$$
and observe that $S_N+S_N=J_N^n$. Consider the set $K_N$ defined as
$$
K_N=\{2b_N+2^{N^2+N/2-1}, 2b_N+2^{N^2+N/2-1}+1,\dots, 2b_N+2^{N^2+N/2}-1\}.
$$
Then $K_N^n \subseteq J_N^n$. Now, if $l=(l_1,\dots,l_n)\in K_N^n$ and $j=(j_1, \dots, j_n)$, then $j$ satisfies $j\in S_N$ and $l-j\in S_N$ if and only if $b_N \leq j_i \leq l_i-b_N$ for every $i=1,\dots, n$. Therefore,
$$
\card \{j\in S_N:~ l-j \in S_N\} \geq \prod_{i=1}^n (l_i-2b_N+1) \gtrsim 2^{n(N^2+N/2)}.
$$
Altogether,
\begin{align*}
&\sum_{l\in S_N +S_N} \left(\operatorname{card}\{j\in S_N: ~l-j \in S_N\}\right)^2\\
&\geq \sum_{l\in K_N^n} \left(\operatorname{card}\{j\in S_N: ~l-j \in S_N\}\right)^2 \\
&\gtrsim \sum_{l\in K_N^n} 2^{2n(N^2+N/2)}
\approx 2^{3n(N^2+N/2)}.
\end{align*}
\end{proof}

\begin{proof}[Proof of Theorem~\ref{T:q=4}]
Let $\psi$ be a smooth function on $\R^n$ supported in the set $\{\xi \in \R^n: ~|\xi|<1/10\}$ such that $0\leq \psi \leq 1$ and $\psi(\xi)=1$ if $|\xi| \leq 1/20$. Let $\{a_l(t)\}_{l\in \Z^n}$ be the sequence of Rademacher functions indexed by the elements of the countable set $\Z^n$. For a given $t\in [0,1]$ we define
\begin{align*}
m_t(\xi,\eta)
&=\sum_{N\in 2\N} 2^{-\frac{nN^2}{2}} \sum_{j\in S_N} \sum_{k\in S_N} a_{j+k}(t) \psi(2^N\xi-j) \psi(2^N\eta-k)\\
&=\sum_{N\in 2\N} F_{t,N}(\xi,\eta),  \quad (\xi, \eta) \in \R^n \times \R^n,
\end{align*}
where the sequence $\{b_N\}_{N\in 2\N}$ appearing in the definition of the sets $I_N$ and $S_N$ is chosen in such a way that the supports of the functions $F_{t,N}$ are pairwise disjoint in $N$. 
Also, let $\varphi$ be a Schwartz function on $\R^n$ whose Fourier transform is supported in the set $\{\xi \in \R^n: ~|\xi|<1/20\}$. For any $N \in 2\N$ we consider the functions $f_N=g_N$ whose Fourier transform satisfies 
$$
\wh{f_N}(\xi)=2^{\frac{nN}{4}-\frac{nN^2}{2}} \sum_{j\in S_N} \wh{\varphi}(2^N \xi-j).
$$
Then, by Plancherel's theorem,
\begin{equation}\label{E:l2}
\|f_N\|^2_{L^2(\R^n)} = \|g_N\|^2_{L^2(\R^n)} \approx 2^{\frac{nN}{2}-nN^2} 2^{-nN} \card{S_N} =1.
\end{equation}

We have
$$
m_t(\xi,\eta) \wh{f_N}(\xi) \wh{g_N}(\eta)
=2^{\frac{nN}{2}-\frac{3nN^2}{2}} \sum_{j\in S_N} \sum_{k\in S_N} a_{j+k}(t) \wh{\varphi}(2^N \xi-j) \wh{\varphi}(2^N \eta-k),
$$
and so
\begin{align*}
&T_{m_t}(f_N,g_N)(x)\\
&=2^{\frac{nN}{2}-\frac{3nN^2}{2}} \sum_{j\in S_N} \sum_{k\in S_N} a_{j+k}(t) 2^{-2nN} \left(\varphi \left(\frac{x}{2^N} \right)\right)^2 e^{2\pi ix \cdot \frac{j+k}{2^N}}\\
&=2^{-\frac{3nN}{2}-\frac{3nN^2}{2}} \left(\varphi \left(\frac{x}{2^N} \right)\right)^2 \sum_{l\in S_N+S_N} a_l(t) e^{2\pi ix \cdot \frac{l}{2^N}} \card \{j\in S_N:~ l-j \in S_N\}.
\end{align*}
By Fubini's theorem and Khintchine's inequality,
\begin{align*}
&\int_0^1 \|T_{m_t}(f_N,g_N)\|_{L^1(\R^n)}\,dt
= \int_{\R^n} \int_0^1 |T_{m_t}(f_N,g_N)(x)|\,dt\,dx\\
&\approx 2^{-\frac{3nN}{2}-\frac{3nN^2}{2}} \int_{\R^n} \left(\varphi \left(\frac{x}{2^N} \right)\right)^2 \left(\sum_{l\in S_N+S_N} (\card \{j\in S_N:~ l-j \in S_N\})^2\right)^{\frac{1}{2}}\,dx\\
&\approx 2^{-\frac{nN}{2}-\frac{3nN^2}{2}} \left(\sum_{l\in S_N+S_N} (\card \{j\in S_N:~ l-j \in S_N\})^2\right)^{\frac{1}{2}}\\
&\gtrsim 2^{\frac{nN}{4}},
\end{align*}
where the last inequality follows from Lemma~\ref{L:S_N}. Therefore, there is $t_N\in [0,1]$ such that 
$$
\|T_{m_{t_N}}(f_N,g_N)\|_{L^1(\R^n)} \gtrsim 2^{\frac{nN}{4}}.
$$

Let $m$ be the function defined as
\begin{equation}\label{E:m}
m=\sum_{N\in 2\N} F_{t_N,N}.
\end{equation}
Then $T_m(f_N, g_N)=T_{m_{t_N}}(f_N,g_N)$ for every $N \in 2\N$, and so
\begin{equation}\label{E:norm}
\|T_m(f_N, g_N)\|_{L^1(\R^n)} \gtrsim 2^{\frac{nN}{4}}.
\end{equation}
A combination of~\eqref{E:l2} and~\eqref{E:norm} thus yields that $T_m$ is not bounded from $L^2(\R^n) \times L^2(\R^n)$ to $L^1(\R^n)$. 

Since $\psi$ is a smooth function and the supports of $\psi(2^N \xi-j)$ are pairwise disjoint in $(N,j)$, we deduce that $m$ is a smooth function. 
To verify that $m$ has bounded partial derivatives of all orders, let us fix a multiindex $\alpha$. Since the functions $F_{{t_N},N}$ have pairwise disjoint supports in $N$, it is enough to verify that the functions $\partial^{\alpha} F_{t_N,N}$ are pointwise bounded by a constant independent of $N$. This is indeed true, as
$$
N|\alpha| = \sqrt{n} N \cdot \frac{|\alpha|}{\sqrt{n}}\leq \frac{nN^2}{2}+\frac{|\alpha|^2}{2n},
$$
which implies
$$
|\partial^{\alpha} F_{t_N,N}(\xi,\eta)|\leq C_{\alpha,\psi} 2^{-\frac{nN^2}{2}+N|\alpha|}
\leq C_{\alpha,\psi} 2^{\frac{|\alpha|^2}{2n}}, \quad (\xi,\eta) \in \R^n \times \R^n.
$$

Finally, let us show that $m$ belongs to $L^4(\R^{2n})$. We have
$$
\|F_{t_N,N}\|^4_{L^4(\R^{2n})} \lesssim 2^{-2nN^2} 2^{-2nN} (\card S_N)^2
\approx 2^{-nN},
$$
and so
$$
\|m\|_{L^4(\R^{2n})}\leq \sum_{N\in 2\N} \|F_{t_N,N}\|_{L^4(\R^{2n})} \lesssim \sum_{N\in 2\N} 2^{-\frac{nN}{4}}<\infty,
$$
as desired.
\end{proof}

\begin{rmk}
As mentioned in Section~\ref{S:introduction}, the function $m$ from the proof of Theorem~\ref{T:q=4} not only belongs to $L^4(\R^{2n})$, but also satisfies the estimate
$$
|\{(\xi,\eta)\in \R^{2n}:~ |m(\xi,\eta)| >\lambda\}|\lesssim \lambda^{-4} \log^{-\alpha}(e/\lambda), \quad \lambda \in (0,1),
$$
for any $\alpha>0$, up to multiplicative constants depending on $\alpha$ and $n$. To verify this, we recall that $m$ is of the form~\eqref{E:m}, where the functions $F_{t_N,N}$ have pairwise disjoint supports in $N \in 2\N$ and satisfy
\begin{align*}
&|\{(\xi,\eta) \in \R^{2n}: ~ |F_{t_N,N}(\xi,\eta)|>\lambda\}|\\
&\lesssim
\begin{cases}
(\card S_N)^2 2^{-2nN} \lesssim 2^{2nN^2-nN} & \text{if } \lambda \in (0,2^{-\frac{nN^2}{2}}),\\
0 & \text{if } \lambda \geq 2^{-\frac{nN^2}{2}}.
\end{cases}
\end{align*}
Therefore,
\begin{align*}
&|\{(\xi,\eta)\in \R^{2n}:~ |m(\xi,\eta)| >\lambda\}|\\
&= \sum_{N\in 2\N} |\{(\xi,\eta) \in \R^{2n}: ~ |F_{t_N,N}(\xi,\eta)|>\lambda\}|\\
&\lesssim
\begin{cases}
\sum_{\substack{N\in 2\N\\ N\leq K}} 2^{2nN^2-nN} \approx 2^{2nK^2-nK} & \text{if } \lambda \in [2^{-\frac{n(K+2)^2}{2}}, 2^{-\frac{nK^2}{2}})\\ & \text{for some } K\in 2\N;\\
0 & \text{if } \lambda \geq 2^{-2n}.
\end{cases}
\end{align*}
We observe that whenever $K\in 2\N$ and $\lambda \in [2^{-\frac{n(K+2)^2}{2}}, 2^{-\frac{nK^2}{2}})$ then
$$
\lambda^{-4} \log^{-\alpha}(e/\lambda) \gtrsim 2^{2nK^2} K^{-2\alpha} \gtrsim 2^{2nK^2-nK},
$$
which yields the desired conclusion.
\end{rmk}

\section{Comparison of Theorem~\ref{T:sobolev} with other multiplier theorems}\label{S:comparison}

In the first part of this section we show that each multiplier satisfying the assumptions of~\cite[Theorem 1.3]{GHS} falls under the scope of Theorem~\ref{T:sobolev} as well.
More precisely, it follows from the estimates below that whenever $m$ is a function belonging to $L^q(\R^{2n})$ for some $1<q<4$ whose partial derivatives up to the order $\lfloor \frac{2n}{4-q} \rfloor +1$ are bounded, then $m\in L^{4}_s(\R^{2n})$ for some $s>\frac{n}{2}$. Thanks to the embedding $L^4_s(\R^{2n}) \hookrightarrow L^{4,\infty}_s(\R^{2n})$, this yields that $m$ satisfies the assumptions of Theorem~\ref{T:sobolev}. The proof of this comparison result is based on a fractional variant of the classical Gagliardo-Nirenberg interpolation inequality due to Brezis and Mironescu~\cite[Theorem 1]{BM}. We start with a few preliminaries on fractional Sobolev spaces.

So far, we have worked with fractional Sobolev spaces defined via the Fourier transform. Let us now introduce another variant of these spaces. 
For any $s>0$ and $1\leq p\leq \infty$, we set $k=\lfloor s \rfloor$ and define the functional
$$
|f|_{W^{s,p}(\R^{2n})}=
\begin{cases}
\|D^k f\|_{L^p(\R^{2n})} & \text{if } s\in \N \\
\int_{\R^{2n}} \int_{\R^{2n}} \frac{|D^k f(x)-D^k f(y)|^p}{|x-y|^{2n+(s-k)p}}\,dx\,dy & \text{if } s\notin \N \text{ and } p<\infty\\
\sup_{x\neq y} \frac{|D^k f(x)-D^k f(y)|}{|x-y|^{s-k}} &\text{if } s\notin \N \text{ and } p=\infty.
\end{cases}
$$ 
The fractional Sobolev space $W^{s,p}(\R^{2n})$ is then defined as the collection of all $k$-times weakly differentiable functions $f$ on $\R^{2n}$ satisfying
$$
\|f\|_{W^{s,p}(\R^{2n})} =\|f\|_{L^p(\R^{2n})} + |f|_{W^{s,p}(\R^{2n})} <\infty.
$$

In general, the space $W^{s,p}(\R^{2n})$ does not coincide with $L^p_s(\R^{2n})$ 
but we have the chain of embeddings
\begin{equation}\label{E:w_l}
W^{s_2,p}(\R^{2n}) \hookrightarrow L^p_s(\R^{2n}) \hookrightarrow W^{s_1,p}(\R^{2n}), 
\end{equation}
where $0<s_1<s<s_2$ and $1<p<\infty$.
The first embedding in~\eqref{E:w_l} thus implies, via Theorem~\ref{T:sobolev} and the embedding $L^4_s(\R^{2n}) \hookrightarrow L^{4,\infty}_s(\R^{2n})$, that inequality 
\begin{equation}\label{E:gagliardo}
\|T_m(f,g)\|_{L^1(\R^n)} \leq C \|m\|_{W^{s,4}(\R^{2n})} \|f\|_{L^2(\R^n)} \|g\|_{L^2(\R^n)}
\end{equation}
is satisfied if $s>n/2$. 

It follows from~\cite[Theorem 1]{BM} that
\begin{equation}\label{E:nirenberg}
\|m\|_{W^{s,4}(\R^{2n})} \lesssim \|m\|_{L^q(\R^{2n})}^{\frac{q}{4}} \|m\|_{W^{\tilde{s},\infty}(\R^{2n})}^{1-\frac{q}{4}}
\end{equation}
holds whenever $s>0$, $1<q<4$ and $\tilde{s}=\frac{4s}{4-q}$. A combination of~\eqref{E:gagliardo} and~\eqref{E:nirenberg} then yields the following fractional variant of~\cite[Theorem 1.3]{GHS}.

\begin{cor}
Let $1<q<4$ and $s>\frac{2n}{4-q}$. Let $m$ be a function on $\R^{2n}$ belonging to $L^q(\R^{2n}) \cap W^{s,\infty}(\R^{2n})$. Then the associated operator $T_m$ admits a bounded extension from $L^2(\R^n) \times L^2(\R^n)$ to $L^1(\R^n)$ and
$$
\|T_m(f,g)\|_{L^1(\R^n)} \leq C \|m\|_{L^q(\R^{2n})}^{\frac{q}{4}} \|m\|_{W^{s,\infty}(\R^{2n})}^{1-\frac{q}{4}} \|f\|_{L^2(\R^n)} \|g\|_{L^2(\R^n)}.
$$
\end{cor}

We finish the paper by briefly commenting on the relationship of Theorem~\ref{T:sobolev} with two other bilinear multiplier theorems. The first one is \cite[Theorem 1.1]{GHS}, which has direct applications to the problem of boundedness of rough bilinear singular integral operators. This theorem gives a sufficient condition for the $L^2(\R^n) \times L^2(\R^n) \to L^1(\R^n)$ boundedness of the operator
\begin{equation}\label{E:sum_operator}
T_{m}^{\Sigma}=\sum_{k\in \Z} T_{m(2^k\cdot)},
\end{equation}
expressed in terms of the membership of the function $m$ (satisfying certain regularity assumptions) into the space $L^q(\R^{2n})$. 
It is shown in~\cite{GHS} that this theorem holds if $q<4$ and fails in the limiting case $q=4$. We emphasize that the counterexample provided there involves test functions that belong to the space $L^4(\R^{2n})$ together with all its partial derivatives. Therefore, the statement of our Theorem~\ref{T:sobolev} is not valid if the operator $T_m$ is replaced by the operator $T_{m}^{\Sigma}$, and we are thus not aware of any direct applications of Theorem~\ref{T:sobolev} to the problem of boundedness of rough bilinear singular integral operators. Finding an analogue of Theorem~\ref{T:sobolev} for the operator $T_{m}^{\Sigma}$ remains an open problem.

The second theorem we would like to focus on is the bilinear variant of the classical Mikhlin-H\"ormander multiplier theorem. Its initial version was obtained by Coifman and Meyer~\cite{CM}, and the topic was investigated further by various authors~\cite{FT,GHH2,GMNT,GMT,GN,GP,GS,MT3,MT2,P,T}. In particular, it is known (see~\cite[Theorem 1]{GHH2}) that whenever
\begin{equation}\label{E:hormander}
\sup_{k\in \Z} \|m(2^k \cdot) \phi\|_{L^4_s(\R^{2n})}<\infty
\end{equation}
holds for some $s>n/2$ then $T_m$ admits a bounded extension from $L^2(\R^n) \times L^2(\R^n)$ to $L^1(\R^n)$.
Here, $\phi$ is a smooth function supported in the annulus $\{\xi\in \R^{2n}:~ 1/2 <|\xi| <2\}$ and satisfying $\sum_{k\in \Z} \phi(2^k \cdot) =1$. Despite certain similarity of this result with the statement of Theorem~\ref{T:sobolev}, which involves the slightly larger Sobolev space $L^{4,\infty}_s(\R^{2n})$ instead of $L^4_s(\R^{2n})$, it turns out that these two multiplier theorems are not comparable. In fact,~\cite[Theorem 1]{GHH2} is incomparable also with the (weaker) variant of Theorem~\ref{T:sobolev} involving the Sobolev space $L^4_s(\R^{2n})$. 
Clearly, any nontrivial constant function satisfies the assumptions of~\cite[Theorem 1]{GHH2} but not of Theorem~\ref{T:sobolev}. On the other hand, \cite[Theorem 1]{GHH2} imposes stronger assumptions on the decay of derivatives of the multiplier near infinity than Theorem~\ref{T:sobolev}. A precise formulation of this fact is the content of the proposition below. It follows from that proposition that if a function of the form~\eqref{E:form} satisfies the assumptions of~\cite[Theorem 1]{GHH2} then the corresponding sequence $c$ necessarily belongs to $\ell^q(\Z^{2n})$ for some $q<4$. Clearly, this requirement is essentially stronger than requiring $c\in \ell^{4}(\Z^{2n})$, which is sufficient for the variant of Theorem~\ref{T:sobolev} involving the Sobolev space $L^4_s(\R^{2n})$. 


We recall that, for any $s>0$, the operator $(-\Delta)^{\frac{s}{2}}$ is defined via the Fourier transform as
$$
[(-\Delta)^{\frac{s}{2}} f]^{\wh{ ~}} (\xi) = (4\pi^2 |\xi|^2)^{\frac{s}{2}} \wh{f}(\xi), \quad \xi \in \R^{2n}.
$$

\begin{prop}\label{P:improving} 
Assume that $m$ is a function on $\R^{2n}$ supported away from the origin and satisfying condition~\eqref{E:hormander} for some $s>n/2$.
Then 
\begin{equation}\label{E:q}
\|(-\Delta)^{\frac{s}{2}} m\|_{L^q(\R^{2n})}<\infty 
\end{equation}
for every $q\in (\max\{\frac{2n}{s},1\},4)$. 
\end{prop}

\begin{proof}
Since $m$ is supported away from the origin, there is $l\in \Z$ such that $m=0$ on the support of $\phi(2^{-k}\cdot)$ when $k<l$. Therefore,
\begin{equation}\label{E:decomposition}
m=\sum_{k=-\infty}^\infty \phi(2^{-k}\cdot) m
=\sum_{k=l}^\infty \phi(2^{-k}\cdot) m.
\end{equation}
Let $q \in (\max\{\frac{2n}{s},1\},4)$ and let $\Phi$ be a smooth function which is equal to one on the annulus $\{\xi \in \R^{2n}:~1/2 <|\xi|<2\}$ and supported in $\{\xi \in \R^{2n}:~1/4 <|\xi|<4\}$. Using~\eqref{E:decomposition}
and the Kato-Ponce inequality~\cite[Theorem 1]{GO}, we obtain
\begin{align*}
&\|(-\Delta)^{\frac{s}{2}} m\|_{L^q(\R^{2n})}
\leq \sum_{k=l}^\infty \|(-\Delta)^{\frac{s}{2}} [\phi(2^{-k}\cdot) m]\|_{L^q(\R^{2n})}\\
&= \sum_{k=l}^\infty \|(-\Delta)^{\frac{s}{2}} [\Phi(2^{-k} \cdot)\phi(2^{-k}\cdot) m]\|_{L^q(\R^{2n})}\\
&\lesssim \sum_{k=l}^\infty \|(-\Delta)^{\frac{s}{2}} [\phi(2^{-k}\cdot) m]\|_{L^4(\R^{2n})} \|\Phi(2^{-k}\cdot)\|_{L^{\frac{4q}{4-q}}(\R^{2n})}\\
&+\sum_{k=l}^\infty \|\phi(2^{-k}\cdot) m\|_{L^{\frac{4q}{4-q}}(\R^{2n})} \|(-\Delta)^{\frac{s}{2}} [\Phi(2^{-k}\cdot)]\|_{L^4(\R^{2n})}\\
&\lesssim \sum_{k=l}^\infty 2^{-k(s-\frac{2n}{q})} \|(-\Delta)^{\frac{s}{2}} [m(2^k \cdot) \phi]\|_{L^4(\R^{2n})} \|\Phi\|_{L^{\frac{4q}{4-q}}(\R^{2n})}\\
&+ \sum_{k=l}^\infty 2^{-k(s-\frac{2n}{q})} \|m(2^k \cdot) \phi\|_{L^{\frac{4q}{4-q}}(\R^{2n})} \|(-\Delta)^{\frac{s}{2}} \Phi\|_{L^4(\R^{2n})}\\
&\lesssim \sup_{k\in \Z} \|(I-\Delta)^{\frac{s}{2}}[m(2^k \cdot) \phi]\|_{L^4(\R^{2n})} +\sup_{k\in \Z} \|m(2^k \cdot)\phi\|_{L^\infty(\R^{2n})}\\
&\lesssim \sup_{k\in \Z} \|(I-\Delta)^{\frac{s}{2}}[m(2^k \cdot) \phi]\|_{L^4(\R^{2n})},
\end{align*}
where the last inequality follows from the fact that the Sobolev space $L^{4}_{s}(\R^{2n})$ is continuously embedded into $L^{\infty}(\R^{2n})$. 
\end{proof}

\section{Acknowledgments}

I would like to thank Akihiko Miyachi for bringing the papers~\cite{KMT} and \cite{KMT2} to my attention. I am also grateful to the referees for their useful comments that helped to improve the paper. 

This research was partially supported by the Hausdorff Center for Mathematics (DFG EXC 2047).


\end{document}